\documentclass[runningheads]{llncs}
\usepackage{graphicx}

\usepackage{amsmath} 
\usepackage{amssymb}  
\usepackage{xcolor}
\usepackage{caption}
\usepackage{graphicx}
\usepackage{mathtools}
\usepackage{caption}
\usepackage{floatrow}
\usepackage{tikz}
\usepackage{mathrsfs}
\usepackage{ulem}

\usepackage{algorithmicx}
    \usepackage[ruled]{algorithm}
    \usepackage{algpseudocode}

\usepackage{pgfplots}
\usepackage{xcolor}
\usetikzlibrary{matrix,arrows,calc,positioning,shapes,decorations.pathreplacing}
\usepackage{graphicx}

\usepackage{amsmath} 

\usepackage{enumerate}
\usepackage[all,tips]{xy}
\SelectTips{cm}{11}

\newcommand{\R}{\mathbb{R}}

\newcommand{\scrA}{\mathscr{A}}
\newcommand{\scrD}{\mathscr{D}}
\newcommand{\scrF}{\mathscr{F}}
\newcommand{\scrG}{\mathscr{G}}

\DeclareMathOperator{\spn}{span}

\DeclareMathOperator{\rank}{rank}

\begin{document}
\title{Virtual Affine Nonholonomic Constraints\thanks{E. Stratoglou and A. Anahory have contributed equally. The authors acknowledge financial support from Grant PID2019-106715GB-C21 funded by MCIN/AEI/ 10.13039/501100011033 and the LINC Global project from CSIC "Wildlife Monitoring Bots" INCGL20022. A.B. was partially supported by NSF grants DMS-1613819 and DMS-2103026, and AFOSR grant FA
9550-22-1-0215.}}

\author{Efstratios Stratoglou\inst{1*} \and
Alexandre Anahory Simoes\inst{2*} \and
Anthony Bloch\inst{3}\and
Leonardo Colombo\inst{4}}
\authorrunning{E. Stratoglou et al.}
%
\institute{Universidad Polit\'ecnica de Madrid (UPM), José Gutiérrez Abascal, 2, 28006 Madrid, Spain.\email{ef.stratoglou@alumnos.upm.es } \and
School of Science and Technology, IE University, Spain.
\email{alexandre.anahory@ie.edu}\\ \and
Department of Mathematics, University of Michigan, Ann Arbor, MI 48109, USA.\\
\email{abloch@umich.edu}\and Centre for Automation and Robotics (CSIC-UPM), Ctra. M300 Campo Real, Km 0,200, Arganda
del Rey - 28500 Madrid, Spain.\email{leonardo.colombo@csic.es}}
\maketitle              
\begin{abstract}
Virtual constraints are relations imposed in a control system that become invariant via feedback, instead of real physical constraints acting on the system.  Nonholonomic systems are mechanical systems with non-integrable constraints on the velocities. In this work, we introduce the notion of \textit{virtual affine nonholonomic constraints} in a geometric framework. More precisely, it is a controlled invariant affine distribution associated with an affine connection mechanical control system. We show the existence and uniqueness of a control law defining a virtual affine nonholonomic constraint.
\end{abstract}
\section{Introduction}
Virtual constraints are relations on the configuration variables of a control system which are imposed through feedback control and the action of actuators, instead of through physical connections such as gears or contact conditions with the environment. The class of virtual holonomic constraints became popular in applications to biped locomotion where it was used to express a desired walking gait, as well as for motion planning to search for periodic orbits and its employment in the technique of transverse linearization to stabilize such orbits.

Virtual nonholonomic constraints are a class of virtual constraints that depend on velocities rather than only on the configurations of the system. Those constraints were introduced in \cite{griffin2015nonholonomic} to design a velocity-based swing foot placement in bipedal robots. In particular, this class of virtual constraints has been used in \cite{hamed2019nonholonomic,horn2020nonholonomic,horn2018hybrid,horn2021nonholonomic} to encode velocity-dependent stable walking gaits via momenta conjugate to the unacatuated degrees of freedom of legged robots and prosthetic legs.  

 The recent work \cite{moran2021energy} intorduces an approach to defining rigorously virtual nonholonomic constraints, but the nonlinear nature of the constraints makes difficult a thorough mathematical analysis. In the work \cite{virtual}, we provide a formal definition of linear virtual nonholonomic constraints, i.e., constraints that are linear on the velocities. Our definition is based on the invariance property under the closed-loop system and coincides with that of \cite{moran2021energy} in the linear case. In this paper we extend the results of \cite{virtual} to the case of affine constraints on the velocities.

In particular, a virtual affine nonholonomic constraint is described by an affine non-integrable distribution on the configuration manifold of the system for which there is a feedback control making it invariant under the flow of the closed-loop system. We provide conditions for the existence and uniqueness of such a feedback law defining the virtual affine nonholonomic constraint.

The remainder of the paper is structured as follows. Section \ref{sec2} introduces nonholonomic systems with affine constraints. We define virtual nonholonomic constraints in Section \ref{sec:controler}, where we provide conditions for the existence and uniqueness of a control law defining a virtual nonholonomic constraint.  In Section \ref{sec4}, we provide an example to illustrate the theoretical results.

\section{Nonholonomic systems with affine constraints}\label{sec2}
We begin with a Lagrangian system on an $n$-dimensional configuration space $Q$ and Lagrangian $L:TQ\to\mathbb{R}$. We assume the Lagrangian has the mechanical form
\begin{equation}\label{lagrangian}
    L=K-V\circ\pi,
\end{equation} where $K$ is a function on $TQ$ describing the kinetic energy of the system, that is, $K=\frac{1}{2}\scrG(v_q, v_q)$, where $\scrG$ is the Riemannian metric on $Q$, $V:Q\to\mathbb{R}$ is a function on $Q$ representing the potential energy, and $\pi:TQ\to Q$ is the tangent bundle projection, locally given by $\pi(q,\dot{q})=q$ with $(q,\dot{q})$ denoting local coordinates on $TQ$. In addition, we denote by $\mathfrak{X}(Q)$ the set of vector fields on $Q$ and by $\Omega^{1}(Q)$ the set of $1$-forms on $Q$. If $X, Y\in\mathfrak{X}(Q),$ then
$[X,Y]$ denotes the standard Lie bracket of vector fields.

Consider affine nonholonomic constraints, that is, for each $q\in Q$ the velocities belong to an affine subspace $\scrA_q$ of the tangent space $T_qQ.$ Thus, $\scrA_q$ can be written as a sum of a vector field $X\in\mathfrak{X}(Q)$ and a nonintegrable distribution $\scrD$ on $Q$, i.e. $\scrA_q=X(q)+\scrD_q$, where $\scrD$ is of constant rank $r$, with $1<r<n$. In this case, we say that the affine space $\scrA_q$ is modelled on the vector subspace $\scrD_q$. In local coordinates $\scrD$ can be expressed as the null space of a $q$-dependent matrix $S(q)$ of dimension $m\times n$ and $\rank S(q)=m,$ with $m=n-r$ as $\scrD_q=\{\dot{q}\in T_qQ : S(q)\dot{q}=0\}$. The rows of $S(q)$ can be represented by the coordinate functions of $m$ independent 1-forms $\mu^b=\mu^b_idq^i$, $1\leq i\leq n, \; 1\leq b\leq m$.
The affine distribution is $\scrA_q=\{\dot{q}\in T_qQ : S(q)(\dot{q}-X(q))=0\}$, hence  $\scrA=\{(q,\dot{q})\in TQ : \phi(q,\dot{q})=0\}$, with $\phi(q,\dot{q})=S(q)\dot{q}+Z(q)$ and $Z(q)=-S(q)X(q)\in\R^m$ (see \cite{Sansonetto}, \cite{Sansonetto2} for instance).

\begin{definition}\label{nonholonomicsystem}
A mechanical system with \textit{affine nonholonomic constraints} on a smooth manifold $Q$ is given
by the triple $(\scrG, V, \scrA)$, where $\scrG$ is
a Riemannian metric on $Q,$ representing the kinetic energy of the
system, $V:Q\rightarrow\mathbb{R}$ is a smooth function representing the potential
energy, and $\scrA$ an affine distribution on $Q$
describing the affine nonholonomic constraints.
\end{definition}

In any Riemannian manifold, there is a unique connection $\nabla^{\scrG}:\mathfrak{X}(Q)\times \mathfrak{X}(Q) \rightarrow \mathfrak{X}(Q)$ called the \textit{Levi-Civita connection} satisfying the following two properties:
\begin{enumerate}
\item $[ X,Y]=\nabla_{X}^{\scrG}Y-\nabla_{Y}^{\scrG}X$ (symmetry)
\item $X(\scrG(Y,Z))=\scrG(\nabla_{X}^{\scrG}(Y,Z)+\scrG(Y,\nabla_{X}^{\scrG}Z)$ (compatibillity of the metric).
\end{enumerate}

The Levi-Civita connection helps us describe the trajectories of a mechanical Lagrangian system. The trajectories $q:I\rightarrow Q$ of a mechanical Lagrangian determined by a Lagrangian function as in \eqref{lagrangian} satisfy the following equations
\begin{equation}\label{ELeq}
    \nabla_{\dot{q}}^{\scrG}\dot{q} + \text{grad}_{\scrG}V(q(t)) = 0,
\end{equation}
where the vector field $\text{grad}_{\scrG}V\in\mathfrak{X}(Q)$ is characterized by $\scrG(\text{grad}_{\scrG}V, X) = dV(X), \; \mbox{ for  every } X \in
\mathfrak{X}(Q)$. Observe that if the potential function vanishes, then the trajectories of the mechanical system are just the geodesics with respect to the connection $\nabla^{\scrG}$.

\section{Virtual affine nonholonomic constraints}\label{sec:controler}

In this section we present a detailed construction of virtual affine nonholonomic constraints. As will be clear from the definition of virtual affine nonholonomic constraints their existence is essentially linked to a controlled system, rather than to the affine distribution $\scrA$ defined by the constraints. As a consequence we give the necessary tools for controlled mechanical systems.

Given a Riemannian metric $\scrG$ on $Q$, we can use its non-degeneracy property to define the musical isomoprhism $\flat:\mathfrak{X}(Q)\rightarrow \Omega^{1}(Q)$ defined by $\flat(X)(Y)=\scrG(X,Y)$ for any $X, Y \in \mathfrak{X}(Q)$. Also, denote by $\sharp:\Omega^{1}(Q)\rightarrow \mathfrak{X}(Q)$ the inverse musical isomorphism, i.e., $\sharp=\flat^{-1}$.

Given an external force $F^{0}:TQ\rightarrow T^{*}Q$ and a control force $F:TQ\times U \rightarrow T^{*}Q$ of the form
\begin{equation}
    F(q,\dot{q},u) = \sum_{a=1}^{m} u_{a}f^{a}(q)
\end{equation}
where $f^{a}\in \Omega^{1}(Q)$ with $m<n$, $U\subset\mathbb{R}^{m}$ the set of controls and $u_a\in\mathbb{R}$ with $1\leq a\leq m$ the control inputs, consider the associated mechanical control system of the form
\begin{equation}\label{mechanical:control:system}
    \nabla^{\scrG}_{\dot{q}(t)} \dot{q}(t) =Y^{0}(q(t),\dot{q}(t)) + u_{a}(t)Y^{a}(q(t)),
\end{equation}
with $Y^{0}=\sharp(F^{0})$ and $Y^{a}=\sharp(f^{a})$ the corresponding force vector fields.

Hence, $q$ is the trajectory of a vector field of the form
\begin{equation}\label{SODE}\Gamma(v_{q})=G(v_{q})+u_{a}(Y^{a})_{v_{q}}^{V},\end{equation}
where $G$ is the vector field determined by the unactuated forced mechanical system
\begin{equation*}
    \nabla^{\scrG}_{\dot{q}(t)} \dot{q}(t) =Y^{0}(q(t),\dot{q}(t))
\end{equation*}
and where the vertical lift of a vector field $X\in \mathfrak{X}(Q)$ to $TQ$ is defined by $$X_{v_{q}}^{V}=\left. \frac{d}{dt}\right|_{t=0} (v_{q} + t X(q)).$$

\begin{definition}
    The distribution $\mathcal{F}\subseteq TQ$ generated by the vector fields  $\sharp(f_{i})$ is called the \textit{input distribution} associated with the mechanical control system \eqref{mechanical:control:system}.
\end{definition}

Now we will introduce the concept of virtual affine nonholonomic constraint.

\begin{definition}
A \textit{virtual affine nonholonomic constraint} associated with the mechanical control system \eqref{mechanical:control:system} is a controlled invariant affine distribution $\scrA\subseteq TQ$ for that system, that is, 
there exists a control function $\hat{u}:\scrA\rightarrow \mathbb{R}^{m}$ such that the solution of the closed-loop system satisfies $\psi_{t}(\scrA)\subseteq \scrA$, where $\psi_{t}:TQ\rightarrow TQ$ denotes its flow.
\end{definition}

Before we proceed to the theorem which gives the necessary conditions for the existence and uniqueness of a control law that turns an affine distribution into a controlled invariant affine distribution (virtual affine nonholonomic constraint), we present some necessary preliminaries.

\begin{definition}
    If $W$ is an affine subspace of the vector space $V$ modelled on the vector subspace $W_{0}$, then the \textit{dimension} of the affine subspace $W$ is defined to be the dimension of the model vector subspace $W_{0}$.
\end{definition}

Two affine subspaces $W_1$ and $W_2$ of a vector space $V$ are \textit{transversal} and we write $W_1\pitchfork W_2$ if 
\begin{enumerate}
    \item $V=W_1+W_2$.
    \item $\dim V=\dim W_1+\dim W_2$, i.e., the dimensions of $W_1$ and $W_2$ are complementary with respect to the ambient space dimension. 
\end{enumerate}
\begin{remark}
    If $W_1$ and $W_2$ are subspaces of $V$ then the previous definition implies that $V=W_{1}\oplus W_{2}$.
\end{remark}
\begin{remark}
    If $W_1$ and $W_2$ are affine spaces modelled on vector subsspaces $W_{10}$ and $W_{20}$, respectively, then $W_{1}\pitchfork W_2$ if and only if $V=W_{10}\oplus W_{20}$.
\end{remark}

\begin{proposition}\label{inhereted transversality_prop}
    For two distributions $\scrA$ and $\scrF$ on a manifold $Q$ where $\scrA$ is an affine distribution with $\scrD$ the associated model distribution and $X$ a vector field on $Q$ satisfying $\scrA=X+\scrD$, the tranversality condition for $\scrA$ and $\scrF$ is an inherited property from the transversality of the model distribution $\scrD$ and vice versa, namely,
    $\scrA \pitchfork \scrF \Leftrightarrow \scrD \pitchfork \scrF$.
\end{proposition}
\begin{proof}
    First consider that $\scrA \pitchfork \scrF$ which means that for every $q\in Q$ we have \[T_qQ=\scrA_q + \scrF_q = X(q)+\scrD_q+\scrF_q\]
    hence, for every $v_q\in T_qQ$ there exist vectors $d_q \in\scrD_q$ and $f_q\in\scrF_q$ such that \[v_q=X(q)+d_q+f_q\;\Leftrightarrow\; v_q-X(q)=d_q+f_q.\]
    Since $v_q-X(q) \in T_qQ$ and $v_q$ is arbitrary we have $T_qQ=\scrD_q+\scrF_q$ for every $q\in Q$. Together with the fact that $\dim \scrD_q+ \dim \scrF_q = \dim T_qQ$, we have that $\scrD\pitchfork\scrF.$
    \\Now suppose that $\scrD\pitchfork\scrF$. Note that this is the same as $TQ = \scrD\oplus\scrF$. Hence, as before, for $v_q\in T_qQ$ there are $d_q \in\scrD_q$ and $f_q\in\scrF_q$ such that \[v_q=d_q+f_q\;\Leftrightarrow\; v_q+X(q)=X(q)+d_q+f_q.\]
    By the same argument as above, $v_q+X(q)\in T_qQ$ and $v_q$ is arbitrary, thus, together with the dimension condition, we conclude that $\scrA\pitchfork\scrF.$ 
\end{proof}

\begin{proposition}\label{TQ=A+F to TTQ=TA+TF_ prop}
    Consider two distributions $\scrA$ and $\scrF$ where the first is an affine distribution as defined previously i.e. $\scrA=X + \scrD$, with $X\in\mathfrak{X}(Q)$ and $\scrD$ its associated model distribution. For $v_q\in\scrA$ we have \[\scrA\pitchfork\scrF \Longrightarrow T_{v_q}(TQ)=T_{v_q}\scrA \oplus \scrF^V_{v_q},\] where $\scrF^V_{v_q}$ is the vertical lift of $\scrF_{v_q}.$ 
\end{proposition}

\begin{proof}

   From the structure of $\scrA$, i.e., from the fact that each $v_q\in\scrA_{q}$ can be written as $v_q=Z(q)+d_q$ where $d_q\in\scrD_q$, we may conclude that $\scrA_{q}$ is a $r$-dimensional manifold, where $r$ is the rank of the distribution $\scrD$. Thus $\scrA$ is a fiber bundle whose base space is the $n$ dimensional manifold $Q$ and whose fibers are $r$ dimensional affine subspaces. Hence,
    \[\dim(T_{v_q}\scrA)=\dim(T_{d_q}\scrD)=n+r\] and since $\dim\scrF^V_{v_q}=n-r=m$, we have that \[ \dim T_{v_q}(TQ)=\dim T_{v_q}\scrA + \dim \scrF^V_{v_q}=n+r + m = 2n.\]
    So, in order to prove that both subspaces are transversal it suffices to prove that their intersection contains only the zero tangent vector.
    Indeed, suppose that $v_{q}\in \scrA$ and $X_{v_{q}} \in T_{v_{q}}\scrA$. Since $\scrA$ is defined to be the set of vectors satisfying the equation $\phi=0$, the tangent vector satisfies $T_{v_{q}}\phi(X_{v_{q}})=0$. If, in addition, $X_{v_{q}} \in \scrF^V_{v_{q}}$, then it can be written as
    $$X_{v_{q}} = c^{i} \sharp(f_{i})^V_{v_{q}}.$$
    However, $$T_{v_{q}}\phi(\sharp(f_{i})^V_{v_{q}})=(S(q)\sharp(f_{i}))^V_{v_{q}}$$ from where it follows that if $c^{i} \sharp(f_{i})^V_{v_{q}}$ was in the null space of the linear map $T_{v_{q}}\phi$, then $c^{i}\sharp(f_{i})$ would be in the null space of $S(q)$ which is false, since these are vectors in $\scrD_{q}$ and $\scrF$ and $\scrD$ are transversal using the previous proposition. Thus $X_{v_{q}}=0$.
    
\end{proof}

\begin{theorem}\label{main theorem}
    If the affine distribution $\scrA$ and the control input distribution $\scrF$ are transversal, then there exists a unique control function making the distribution a virtual affine nonholonomic constraint associated with the mechanical control system \eqref{mechanical:control:system}.
\end{theorem}

\begin{proof}
    Suppose that $\scrA\pitchfork\scrF$ and that trajectories of the contol system \eqref{mechanical:control:system} may be written as the integral curves of the vector field $\Gamma$ defined by \eqref{SODE}. From Proposition \ref{TQ=A+F to TTQ=TA+TF_ prop} we have  $$T_{v_q}(TQ)=T_{v_q}\scrA \oplus \scrF^V_{v_q},$$ where $v_q\in\scrA$ and $\scrF^V_{v_q}=\spn\{(Y^a)^V_{v_q}\}$.
    Using the uniqueness decomposition property arising from transversality, we conclude there exists a unique vector $\tau^{*}(v_{q})=(\tau_{1}^{*}(v_{q}),\cdots, \tau_{m}^{*}(v_{q}))\in \mathbb{R}^{m}$ such that $\Gamma(v_{q})=G(v_{q})+\tau_{a}^{*}(v_{q})(Y^{a})_{v_{q}}^{V}\in T_{v_{q}}\scrA.$ Next, we show that $\Gamma$ depends smoothly on $v_q.$ If $\scrA$ is defined by $m$ constraints of the form $\phi^{b}(v_{q})=0$, \;$1\leq b\leq m$, then the condition above may be rewritten as $d\phi^{b}(G(v_{q})+\tau_{a}^{*}(v_{q})(Y^{a})_{v_{q}}^{V})=0,$ which is equivalent to $$\tau_{a}^{*}(v_{q})d\phi^{b}((Y^{a})_{v_{q}}^{V})=-d\phi^{b}(G(v_{q})).$$ Note that, the equation above is a linear equation of the form $P(v_{q})\tau=b(v_{q})$, where $b(v_{q})$ is the vector $(-d\phi^{1}(\Gamma(v_{q})), \dots, -d\phi^{m}(\Gamma(v_{q})))\in \mathbb{R}^{m}$ and $P(v_{q})$ is the $m\times m$ matrix with entries $P^{b}_{a}(v_{q})=d\phi^{b}((Y^{a})_{v_{q}}^{V})=\mu^{b}(q)(Y^{a})$, where the last equality may be deduced by computing the expressions in local coordinates. That is, if $(q^{i} \dot{q}^{i})$ are natural bundle coordinates for the tangent bundle, then

    \begin{equation*}
        \begin{split}
            d\phi^{b}((Y^{a})_{v_{q}}^{V}) & = \left(\frac{\partial \mu^{b}_{i}}{\partial q^{j}}\dot{q}^{i}dq^{j} + \frac{\partial Z_i}{\partial q^j}dq^j + \mu^{b}_{i}d\dot{q}^{i}\right)\left(Y^{a,k}\frac{\partial}{\partial \dot{q}^{k}}\right) \\
            & = \mu^{b}_{i}Y^{a,i} = \mu^{b}(q)(Y^{a}).
        \end{split}
    \end{equation*}
    
    In addition, $P(v_{q})$ has full rank, since its columns are linearly independent. In fact suppose that
    \begin{equation*}
        c_{1}\begin{bmatrix} \mu^{1}(Y^{1}) \\
        \vdots \\
        \mu^{m}(Y^{1}) \end{bmatrix} + \cdots + c_{m}\begin{bmatrix} \mu^{1}(Y^{m}) \\
        \vdots \\
        \mu^{m}(Y^{m}) \end{bmatrix}= 0,
    \end{equation*}
    which is equivalent to
    \begin{equation*}
        \begin{bmatrix} \mu^{1}(c_{1}Y^{1}+\cdots + c_{m}Y^{m}) \\
        \vdots \\
        \mu^{m}(c_{1}Y^{1}+\cdots + c_{m}Y^{m}) \end{bmatrix}.
    \end{equation*}
    However, from Proposition \ref{inhereted transversality_prop} we have $\scrD\cap \scrF = \{0\}$ which implies that $c_{1}Y^{1}+\cdots + c_{m}Y^{m}=0$. Since $\{Y_{i}\}$ are linearly independent we conclude that $c_{1}=\cdots=c_{m}=0$ and $P$ has full rank. But, since $P$ is an $m\times m$ matrix, and $\scrD$ is a regular distribution, it must be invertible. Therefore, there is a unique vector $\tau^{*}(v_{q})$ satisfying the matrix equation and $\tau^{*}:\scrD\rightarrow \mathbb{R}^{m}$ is smooth since it is the solution of a matrix equation depending smoothly on $v_{q}$. Hence, $\Gamma$ is a smooth vector field tangent to $\scrA$ and its flow remains in $\scrA.$\hfill$\square$
\end{proof}

\section{An example}\label{sec4}
Consider a boat with a payload on the see with a position-dependent stream. The position of the boat's center of mass is modeled by the configuration manifold $\mathbb{R}^{2}$ to which we add an orientation to obtain a complete description of its location in space, so that the system total configuration manifold is $\mathbb{R}^{2}\times \mathbb{S}$ with local coordinates $q=(x,y,\theta)$. The sea's current is modeled by the vector field $C:\mathbb{R}^{2}\rightarrow \mathbb{R}^{2}$, $C=(C^{1}(x,y), C^{2}(x,y))$.

The boat is well modeled by a forced mechanical system with Lagrangian function
$L= \frac{m}{2}(\dot{x}^{2} + \dot{y}^{2}) + \frac{I}{2}\dot{\theta}^{2},$
where $m$ is the boat's mass, $I$ is the moment of inertia, and the external force is denoted by
$F^{ext} = W^{1} dx + W^{2} dy$
accounting for the action of the current on the center of mass of the boat and to which we add a control force
$F=u(\sin \theta dx - \cos \theta dy + d\theta ).$

The functions $W^{1}$ and $W^{2}$ are defined according to
$$\begin{cases}
    W^{1} &=  m  \ d \left(\sin^{2}\theta C^{1} - \sin\theta\cos\theta C^{2}\right)(\dot{q}) \\
    W^{2} &= m \ d \left(-\sin\theta\cos\theta C^{1} + \cos^{2}\theta C^{2} \right)(\dot{q}).
\end{cases},$$
where $d$ represents the differential of the functions inside the parenthesis. The external force assures that in the absence of controls, the dynamics of the boat satisfies the following kinematic equations
$$\begin{cases}
    \dot{x} = & \sin^{2}\theta C^{1} - \sin\theta\cos\theta C^{2} \\
    \dot{y} = & -\sin\theta\cos\theta C^{1} + \cos^{2}\theta C^{2},
\end{cases}$$
whenever the initial velocities in the $x$ and $y$ direction vanish. The corresponding controlled forced Lagrangian system is
\begin{equation*}
    m\ddot{x}=u \sin\theta + W^{1}, \quad m\ddot{y}=-u \cos\theta + W^{2}, \quad I\ddot{\theta}=u,
\end{equation*}
and, as we will show, it has the following virtual affine nonholonomic constraint
$$\sin\theta \dot{x} - \cos\theta \dot{y}= C^{2}\cos\theta - C^{1}\sin\theta.$$
The input distribution $\mathcal{F}$ is generated just by one vector field $$Y=\frac{\sin \theta}{m}\frac{\partial}{\partial x}-\frac{\cos \theta}{m}\frac{\partial}{\partial y}+\frac{1}{I}\frac{\partial}{\partial \theta},$$
while the virtual nonholonomic constraint is the affine space $\mathcal{A}$ modelled on the distribution $\mathcal{D}$ defined as the set of tangent vectors $v_{q}\in T_{q}Q$  where $\mu(q)(v)=0,$ with $\mu=\sin\theta dx - cos\theta dy$. Thus, we may write it as
$$\mathcal{D}=\hbox{span}\Big{\{} X_{1}=\cos \theta\frac{\partial}{\partial x} + \sin\theta \frac{\partial}{\partial y},\, X_{2}=\frac{\partial}{\partial \theta}\Big{ \}}.$$
The affine space is given as the zero set of the function $\phi(q,v) = \mu(q)(v) + Z(q)$ with
$Z(q) = \cos \theta C^{2}(x,y) - \sin \theta C^{1}(x,y)$
or, equivalently, as the set of vectors $v_{q}$ satisfying $v_{q}-C(q)\in \mathcal{D}_{q}$.

We may check that $\mathcal{A}$ is controlled invariant for the controlled Lagrangian system above. In fact, the control law
$\hat{u}(x,y,\theta,\dot{x},\dot{y},\dot{\theta})=-m\dot{\theta}(\cos\theta \dot{x} +\sin \theta \dot{y})$
makes the affine space invariant under the closed-loop system, since in this case, the dynamical vector field arising from the controlled Euler-Lagrange equations given by
$$\Gamma = \dot{x}\frac{\partial}{\partial x} + \dot{y}\frac{\partial}{\partial y} + \dot{\theta}\frac{\partial}{\partial \theta} + \left(\frac{\hat{u}\sin \theta + W^{1}}{m}\right)\frac{\partial}{\partial \dot{x}} + \left( - \frac{\hat{u}\cos \theta - W^{2}}{m} \right)\frac{\partial}{\partial \dot{y}} + \frac{\hat{u}}{I}\frac{\partial}{\partial \dot{\theta}}$$
is tangent to $\mathcal{A}$. This is deduced from the fact that $$\Gamma(\sin\theta \dot{x} - \cos\theta \dot{y} + \cos \theta C^{2}(x,y) - \sin \theta C^{1}(x,y))=0.$$


\begin{thebibliography}{8}
\bibitem{virtual}
A. Anahory Simoes, E. Stratoglou, A. Bloch, L. Colombo. Virtual Nonholonomic Constraints: A Geometric Approach. arXiv e-prints, arXiv: 2207.01299 (2022)

\bibitem{Sansonetto} F. Fasso, N. Sansonetto. Conservation of energy and momenta in nonholonomic systems with affine constraints. Regular and Chaotic Dynamics, 20(4), 449-462 (2015).

\bibitem{Sansonetto2}F. Fasso, L. García-Naranjo, N. Sansonetto. Moving energies as first integrals of nonholonomic systems with affine constraints. Nonlinearity, 31(3), 755 (2018).

\bibitem{griffin2015nonholonomic} B. Griffin, J. Grizzle
Nonholonomic virtual constraints for dynamic walking. 54th IEEE Conference on Decision and Control 4053--4060 (2015).

\bibitem{hamed2019nonholonomic} K. Hamed, A. Ames. Nonholonomic hybrid zero dynamics for the stabilization of periodic orbits: Application to underactuated robotic walking. IEEE Transactions on Control Systems Technology,
28(6), 2689--2696 (2019).


\bibitem{horn2020nonholonomic}
J. Horn, A. Mohammadi, K. Hamed, R. Gregg. Nonholonomic virtual constraint design for variable-incline bipedal robotic walking. IEEE Robotics and Automation Letters, 5(2),
3691--3698 (2020).

\bibitem{horn2018hybrid} J. Horn, A. Mohammadi, K. Hamed, R. Gregg. Hybrid zero dynamics of bipedal robots under nonholonomic virtual constraints. IEEE Control Systems Letters, 3(2), 386--391 (2018).

\bibitem{horn2021nonholonomic} J. Horn, R.Gregg. Nonholonomic Virtual Constraints for Control of Powered Prostheses Across Walking Speeds. IEEE Transactions on Control Systems Technology. (2021).


\bibitem{moran2021energy}
A. Moran-MacDonald. Energy injection for mechanical systems through the method of Virtual Nonholonomic Constraints. University of Toronto (2021).



\end{thebibliography}
\end{document}